\documentclass[12pt]{amsart}
\usepackage{amscd}
\usepackage{amsfonts}
\usepackage[all]{xy}
\usepackage{amsmath,amsthm,hyperref}
\usepackage{amsmath,amssymb,amsthm,latexsym}
\usepackage{amscd}
\usepackage{amsmath}
\numberwithin{equation}{section}
\usepackage{setspace}
\usepackage{caption}
\newtheorem{theorem}{Theorem}[section]

\newtheorem{corollary}[theorem]{Corollary}
\newtheorem{lemma}[theorem]{Lemma}

\theoremstyle{remark}
\newtheorem{remark}[theorem]{Remark}
\newtheorem{example}[theorem]{\bf Example}
\newcommand{\R}{\mathbb{R}}
\newcommand{\C}{\mathbb{C}}

\begin{document}

\title [Minimal surfaces in $\R^n$ via potentials]{\bf{ Willmore surfaces in spheres via loop groups III:  on minimal surfaces in space forms}}
\author{Peng Wang}
 
%\date{}
\maketitle

\begin{center}
{\bf Abstract}
\end{center}

\begin{abstract}
 The family of Willmore immersions from a Riemann surface into $S^{n+2}$ can be divided naturally into the subfamily of Willmore surfaces conformally equivalent to a minimal surface in $\R^{n+2}$ and those which are not conformally  equivalent to a minimal surface in $\R^{n+2}$. On the level of their conformal Gauss maps into $Gr_{1,3}(\R^{1,n+3})=SO^+(1,n+3)/SO^+(1,3)\times SO(n)$ these two classes of Willmore immersions into $S^{n+2}$ correspond to conformally harmonic maps for which every image point, considered as a 4-dimensional Lorentzian subspace of $\R^{1,n+3}$, contains a fixed lightlike vector  or
where it does not contain such a ``constant lightlike vector". Using the loop group formalism for the construction of Willmore immersions we characterize in this paper precisely those normalized potentials which correspond to conformally harmonic maps containing a lightlike vector. Since the special form of these potentials can easily be avoided, we also precisely characterize those potentials which produce Willmore immersions into $ S^{n+2}$ which are not conformal to a minimal surface in $\R^{n+2}$. It turns out that our proof also works analogously for minimal immersions into the other space forms.
\end{abstract}



\section{Introduction}

This is the third paper of a series of papers concerning the global geometry of Willmore surfaces in  terms of loop group theory. Our aim is to derive a criterion  characterizing the normalized potentials of all strongly conformally harmonic maps which either correspond to minimal surfaces in $\R^{n+2}$ or to no Willmore surfaces at all.  On the one hand, this provides a characterization  of minimal surfaces in $\R^{n+2}$, as well as this special class of strongly conformally harmonic maps. On the other hand, it also allows us to derive Willmore surfaces different from minimal surfaces in $\R^{n+2}$
by excluding this special type of normalized potentials. This is particularly important for the application of the main results of \cite{DoWa1} to generic Willmore surfaces in $S^{n+2}$.

It is well-known that minimal surfaces in Riemannian space forms provide standard examples of Willmore surfaces \cite{Bryant1984}, \cite{Bryant1988}, \cite{Weiner}. To pick up all minimal surfaces in space forms among Willmore surfaces, we provide a description of them via potentials.

To be concrete, let  $\mathcal{F}$ be a harmonic map from a Riemann surface $M$ into $Gr_{1,3}(\R^{1,n+3})$ ($=SO^+(1,n+3)/SO^+(1,3)\times SO(n)$), with a lift $F:M\rightarrow SO^+(1,n+3)$ and M-C form
\[\alpha =F^{-1}dF=\left(
                      \begin{array}{cc}
                        A_1 &  B_1 \\
                        -B_1^tI_{1,3} &  A_2 \\
                      \end{array}
                    \right)dz+\left(
                      \begin{array}{cc}
                        \bar{A}_1  &  \bar{B}_1 \\
                        -\bar{B}_1^tI_{1,3} & \bar{A}_2  \\
                      \end{array}
                    \right)d\bar{z}.\]
Here \[A_1\in Mat(4\times 4, \C),\ A_2\in Mat(n\times n, \C),\ B_1\in Mat(4\times n, \C),\ I_{1,3}=\hbox{diag}(-1,1,1,1) .\]
 $\mathcal{F}$ is called strongly conformal if $B_1$ satisfies $B_1^tI_{1,3}B_1=0$. Note that this condition is independent of the choice of $F$ \cite{DoWa1}. We also recall  that the conformal Gauss map of a Willmore surface is a strongly conformally harmonic map \cite{DoWa1}. Conversely, by Theorem 3.10 of \cite{DoWa1}, there are two different kinds of strongly conformally harmonic maps:

{\vspace{2mm}\em   Those  which contain a constant lightlike vector and those which do not contain a constant lightlike vector.
\vspace{2mm}} \\
Moreover from Theorem 3.10 of \cite{DoWa1}, we see that if a strongly conformally  harmonic map $\mathcal{F}$ does not contain a lightlike vector, $f$ will always be the conformal Gauss map of some Willmore map.
 And, most importantly, these Willmore maps correspond  exactly to all the Willmore maps which are {\em not} M\"{o}bius equivalent to any minimal surface in $\mathbb{R}^{n+2}$, since minimal surfaces in $\mathbb{R}^{n+2}$ can be characterized as  those Willmore surfaces whose conformal Gauss map contains a constant  lightlike vector \cite{Helein}, \cite{Xia-Shen}, \cite{Ma-W1} (See Lemma 1.2 below, see also \cite{Bryant1984}, \cite{Ejiri1988}, \cite{Mus1}, \cite{Mon} and \cite{BFLPP}). Since minimal surfaces in  $\mathbb{R}^{n+2}$  can be constructed by a straightforward way, one will be mainly interested in Willmore surfaces {\em not } M\"{o}bius equivalent to minimal surfaces in $\mathbb{R}^{n+2}$.  It is therefore vital to derive a  criterion determining whether a strongly conformally harmonic map $f$ contains a lightlike vector or not. Note that this will also yield an interesting description of minimal surfaces in $\R^{n+2}$.  Applying Wu's formula, one will obtain the following description of the normalized potential of $\mathcal{F}$ when it contains a constant light-like vector.

\begin{theorem}\label{th-potential-light}Let $\mathbb{D}$ denote the Riemann surface $S^2,$ $\mathbb{C}$ or the unit disk of $\mathbb{C}$. Let $\mathcal{F}:  \mathbb{D}\rightarrow SO^+(1,n+3)/SO^+(1,3)\times SO(n)$ be a strongly conformally harmonic map which contains a constant light-like vector. Assume that $f(p)=I_{n+4}\cdot K$ w.r.t some base point $p\in\mathbb{D}$ and $z$ is a local coordinate with $z(p)=0$. Then the normalized potential of $f$ with reference point $p$ is of the form
 \begin{equation}\label{eq-w-minimal}
\eta=\lambda^{-1}\left(
                     \begin{array}{cc}
                       0 & \hat{B}_1 \\
                       -\hat{B}^{t}_1I_{1,3} & 0 \\
                     \end{array}
                   \right)dz,\ \hbox{ where }\ \hat{B}_{1}=\left(
\begin{array}{cccc}
 \hat{f}_{11} & \hat{f}_{12} & \ldots &  \hat{f}_{1n} \\
 -\hat{f}_{11} &  -\hat{f}_{12} & \ldots &  -\hat{f}_{1n} \\
 \hat{f}_{31} &\hat{f}_{32} & \ldots &  \hat{f}_{3n} \\
 i\hat{f}_{31} & i\hat{f}_{32} & \ldots &  i\hat{f}_{3n} \\
 \end{array}
\right).\end{equation}
Here the functions $f_{ij}$ are meromorphic  functions on $\mathbb{D}$.

Moreover, $\mathcal{F}$ is of finite uniton type with maximal uniton number $\leq 2$.\end{theorem}

It is well-known that minimal surfaces in Riemannian space forms can be characterized by the following lemma (The statements and proofs can be found in \cite{Helein}, see also \cite{Ma-W1} for a proof of Case (1)).
\begin{lemma} \cite{Helein}, \cite{Xia-Shen} Let $y:M\rightarrow S^{n+2}$ be a Willmore surface, with $\mathcal{F}$ as its conformal Gauss map. We say that $\mathcal{F}$ contains a constant vector $a\in\R^{1,n+3}$ if for any
$p\in M$, $a$ is in the $4-$dim Lorentzian subspace $\mathcal{F}(p)$. Then
\begin{enumerate}
\item $y$ is M\"{o}bius equivalent to a minimal surface in $\mathbb{R}^{n+2}$ if and only if $\mathcal{F}$ contains a non-zero constant lightlike vector.
\item $y$ is M\"{o}bius equivalent to a minimal surface in some $S^{n+2}(c)$ if and only if $\mathcal{F}$ contains a non-zero constant timelike vector.
\item $y$ is M\"{o}bius equivalent to a minimal surface in $\mathbb{H}^{n+2}(c)$ if and only if $\mathcal{F}$ contains a non-zero constant spacelike vector.
\end{enumerate}
 \end{lemma}
 Applying this lemma and Wu's formula, one obtains   the following descriptions of minimal surfaces in space forms.
\begin{theorem}\label{th-mini-spaceform}
Let $\mathcal{F}:M\rightarrow SO^+(1,n+3)/SO(1,3)\times SO(n)$ be a strongly conformally harmonic map. Let
\begin{equation}\label{eq-potential}
    \eta=\lambda^{-1}\left(
    \begin{array}{cc}
      0 & \hat{B}_1 \\
      -\hat{B}_1^tI_{1,3} & 0 \\
    \end{array}
  \right)dz, \ \hat{B}_{1}=\left(
\begin{array}{cccc}
 \mathbf{v}_1 & \mathbf{v}_2 & \ldots &  \mathbf{v}_n\\
 \end{array}
\right),\
\end{equation}
be the normalized potential of $\mathcal{F}$ with respect to some base point $z_0$. Then, up to a conjugation by some $T\in O^+(1,3)\times O(n)$,
\begin{enumerate}
\item $\mathcal{F}$ contains a constant lightlike vector, if and only if every $\mathbf{v}_j$ has the form
\begin{equation}\label{eq-potential-r}
 \mathbf{v}_j=
f_{j0}\left(
                                      \begin{array}{cccc}
                                        f_{j1} &
                                        -f_{j1} &
                                        f_{j3} &
                                        if_{j3} \\
                                      \end{array}
                                    \right)^t, \hbox{ with }  f_{jl}  \hbox{ meromorphic}.
\end{equation}

\item $\mathcal{F}$ contains a constant timelike vector, if and only if every $\mathbf{v}_j$ has the form
\begin{equation}\label{eq-potential-s}
\mathbf{v}_j=
g_{j}\left(
                                      \begin{array}{cccc}
                                        0 &
                                        2g_{0}&
                                        1-g_{0}^2 &
                                        i(1+g_{0}^2)\\
                                      \end{array}
                                    \right)^t, \hbox{ with }  g_{j},\ g_0  \hbox{ meromorphic}.
\end{equation}

\item $\mathcal{F}$ contains a constant spacelike vector, if and only if every $\mathbf{v}_j$ has the form
\begin{equation}\label{eq-potential-h}
 \mathbf{v}_j=
h_{j}\left(
                                      \begin{array}{cccc}
                                         2ih_{0} &
                                        0 &
                                        1-h_{0}^2 &
                                        i(1+h_{0}^2)\\
                                      \end{array}
                                    \right)^t, \hbox{ with }  h_{j},\ h_0  \hbox{ meromorphic}.
\end{equation}
\end{enumerate}
Potentials of above form will  be called  canonical potentials for the  corresponding minimal surfaces in space forms.
\end{theorem}

The proof of Case (1) of this is more elaborate than the proof for the other cases. It also has a more general background.  The classical theorem of Bryant \cite{Bryant1984} tells us that every Willmore two-sphere in $S^3$ is M\"{o}bius equivalent to some minimal surface with embedded planer ends in $\R^3$. However, when the co-dimension increases, this does no longer hold true and there are Willmore two-spheres in $S^4$ different  from minimal surfaces in $\R^4$ \cite{Ejiri1988}. In \cite{Wang-1}, by using the loop  group methods developed in \cite{DoWa1}, \cite{DoWa2}, we provide a classification of Willmore two-spheres in $S^{n+2}$ via the \emph { normalized potentials} of their harmonic conformal Gauss maps. The basic idea is to combine the work of Burstall and Guest \cite{BuGu}, \cite{Gu2002} with the DPW method  \cite{DoWa2} and to characterize harmonic conformal Gauss maps of Willmore surfaces by describing their normalized  potentials. We would like to point out that these potentials take values in nilpotent Lie sub-algebras by the results of \cite{BuGu} and \cite{DoWa2}. In \cite{Wang-1} it was shown that there are  $m-2$ different types of such nilpotent Lie sub-algebras  when $n+4=2m$. The possible forms of normalized potentials are also listed explicitly.

 To derive concrete expressions for Willmore surfaces, and to understand their geometric properties one needs to perform Iwasawa decompositions of the meromorphic frames of the given normalized potentials. Since the potentials take values in some nilpotent Lie sub-algebra, the meromorphic frames, i.e.,  ODE solutions with the normalized potentials as coefficient matrices, are Laurant polynomials in $\lambda\in S^1$ \cite{BuGu}, \cite{Gu2002}, \cite{DoWa2}. A theoretical procedure  for the Iwasawa decompositions of such algebraic elements in a loop group  has been presented in \cite{CG}. However,  for the first type of normalized potentials in the classification theorem of \cite{Wang-1}, the Iwasawa decomposition can be obtained in an easier and more straightforward way.

 To be concrete, assume that the normalized potential is of the form (See Section 3 of \cite{DoWa1}, Section 2 of \cite{Wang-1} for the definitions and notations)
 \begin{equation}\label{eq-min-np-0}\eta=\lambda^{-1}\eta_{-1}dz
                    \end{equation}
                    with
                    \begin{equation}\label{eq-min-np} \eta_{-1}=\left(
                      \begin{array}{cc}
                        0 & \hat{B}_1 \\
                        -\hat{B}_1^tI_{1,3} & 0 \\
                      \end{array}
                    \right),~~ \hbox{ and }\  \hat{B}_1=
                                        \left(
                                          \begin{array}{ccccc}
                                            f_{11} & f_{12}&\ldots & f_{m-2,1} & f_{m-2,2} \\
                                            f_{11} & f_{12}&\ldots & f_{m-2,1} & f_{m-2,2} \\
                                            f_{13} & f_{14}&\ldots & f_{m-2,3} & f_{m-2,4} \\
                                            if_{13} & if_{14}&\ldots & if_{m-2,3} & if_{m-2,4} \\
                                          \end{array}
                                        \right),
                                        \end{equation}
where $f_{ij}$ are meromorphic functions on the Riemann surface $\tilde{M}$. Note that this $\eta$ is conjugate to the one in \eqref{eq-w-minimal} by $\tilde{T}=\hbox{diag}(1,-1,1,\ldots, 1)$. So the corresponding harmonic maps are M\"{o}bius equivalent to each other. Moreover,  we have the following theorem, which is in fact part of (1) of Theorem \ref{th-mini-spaceform}.
\begin{theorem}\label{th-minimal} Let $\mathcal{F}: \mathbb{D}\rightarrow SO^+(1,2m-1)/SO^+(1,3)\times SO(2m-4) $ be a strongly confomally harmonic map
 with its normalized potential being of the form in \eqref{eq-min-np}. Then $\mathcal{F}$ contains a constant light-like vector.

 Moreover, if $\mathcal{F}$ is the conformal Gauss map of a strong Willmore map $y: \mathbb{D}\rightarrow S^{2m-2}$, then $rank(\hat{B}_1)\leq 1$ and $y$ is M\"{o}bius equivalent to a minimal surface in $\mathbb{R}^{2m-2}$.
\end{theorem}

It is easy to verify that such $\mathcal{F}$ is of finite uniton type. Moreover, $\mathcal{F}$ actually belongs to one of the simplest cases, called $S^1-invariant$ (See \cite{BuGu}, \cite{Do-Es}, \cite{Wang-1}). For such harmonic maps, by using a straightforward and lengthy computation, one can derive the harmonic map explicitly and then read off all needed information, which will provide a  proof of Theorem \ref{th-minimal}.

\begin{corollary}\label{cor-min} Let $\mathcal{F}: \mathbb{D}\rightarrow SO^+(1,n+3)/SO^+(1,3)\times SO(n)$ be a strongly conformally harmonic map
with its normalized potential $\eta$ of the form \eqref{eq-w-minimal} and of maximal $rank(\hat B_1)=2$. Then $\mathcal{F}$ can not be
the conformal Gauss map of a Willmore surface. In particular, there exist conformally harmonic maps which are not related to any Willmore map.
\end{corollary}

As a consequence, at this point we have obtained a complete description of the strongly conformally harmonic maps which produce either minimal surfaces in $\R^{n+2}$ or no Willmore surfaces at all, which make our theory workable for the study of non-Euclidean-minimal Willmore surfaces. This corollary shows that the characterization theorems here do make sense for the theory in \cite{DoWa1} to deal with global Willmore surfaces different from minimal surfaces in $\R^{n+2}$.\\

This paper is organized as follows. Section 2 provides the form of the potentials of strongly conformally harmonic maps containing a constant lightlike vector. The proofs of Cases (2) and (3) of Theorem \ref{th-mini-spaceform} are also derived in this section. Section 3 contains the characterizations of minimal surfaces in $\R^{n+2}$ in terms of potentials and several technical lemmas providing a proof of our main theorem. The proofs of these technical lemmas are derived in Section 4.

For simplicity we will, in this paper,  always retain the notation of  \cite{DoWa1} and \cite{Wang-1}. For more details we also refer to  \cite{DoWa1} and \cite{Wang-1}.

\section{Potentials of strongly conformally harmonic maps containing a constant lightlike vector}

This section is to derive the forms of the normalized potentials of strongly conformally harmonic maps containing a non-zero constant real vector. The basic idea is to characterize the Maurer-Cartan form of such a strongly conformally harmonic map $\mathcal{F}:  \mathbb{D}\rightarrow SO^+(1,n+3)/SO^+(1,3)\times SO(n)$.
\subsection{Proof of Theorem \ref{th-potential-light}}
The proof of  Theorem \ref{th-potential-light}
relies on the following technical lemma.
\begin{lemma}\label{lemma-ode} Let $$A_1=\left(
        \begin{array}{cccc}
          0 & 0 & a_{13} & a_{14} \\
          0 & 0 & -a_{13} & -a_{14} \\
          a_{13} & a_{13} & 0 & a_{34} \\
          a_{14} & a_{14} & -a_{34} & 0 \\
        \end{array}
      \right)$$
      be a holomorphic matrix function on a contractible open Riemann surface $U$.
Let $F_{01}$ be a solution to the equation
 \begin{equation*}F_{01}^{-1}dF_{01}=A_{1}dz,\ ~~~~  F_{01}|_{z=0}=I_4.\end{equation*}
 Then  \[ F_{01}=\left(
              \begin{array}{cccc}
                1+\frac{1}{2}(b_{13}^2+b_{14}^2) & \frac{1}{2}(b_{13}^2+b_{14}^2)  & b_{13} & b_{14} \\
                 -\frac{1}{2}(b_{13}^2+b_{14}^2) & 1-\frac{1}{2}(b_{13}^2+b_{14}^2)  & -b_{13} & -b_{14} \\
                 b_{13}  &  b_{13}   & 1 & 0 \\
                 b_{14}&  b_{14}   & 0 & 1 \\
              \end{array}
            \right)\left(
                          \begin{array}{cccc}
                            1 &   &   &  \\
                              & 1 &  &   \\
                              &  &\cos \varphi & \sin\varphi \\
                              &  & -\sin\varphi &\cos \varphi \\
                          \end{array}
                        \right)\]
with
\[\varphi =\int_0^z  a_{34} \ dw \]
and
 \[b_{13} =\int_0^z (a_{13}\cos \varphi+a_{14}\sin \varphi ) dz ,\    ~ \ b_{14} = \int_0^z (-a_{13}\sin \varphi+a_{14}\cos \varphi) dz.\]
\end{lemma}

\begin{proof} Set \[\tilde{F}_{01}=\left(
                          \begin{array}{cccc}
                            1 &   &   &  \\
                              & 1 &  &   \\
                              &  &\cos \varphi & \sin\varphi \\
                              &  & -\sin\varphi &\cos \varphi \\
                          \end{array}
                        \right),\]
and
\[       \hat{F}_{01}=\left(
              \begin{array}{cccc}
                1+\frac{1}{2}(b_{13}^2+b_{14}^2) & \frac{1}{2}(b_{13}^2+b_{14}^2)  & b_{13} & b_{14} \\
                 -\frac{1}{2}(b_{13}^2+b_{14}^2) & 1-\frac{1}{2}(b_{13}^2+b_{14}^2)  & -b_{13} & -b_{14} \\
                 b_{13}  &  b_{13}   & 1 & 0 \\
                 b_{14}&  b_{14}   & 0 & 1 \\
              \end{array}
            \right).\]
             Straightforward computations yield
\[\tilde{F}_{01}^{-1} d\tilde{F}_{01}=\left(
        \begin{array}{cccc}
          0 & 0 & 0 & 0 \\
          0 & 0 & 0 & 0 \\
          0 & 0 & 0 & a_{34} \\
          0 & 0 & -a_{34} & 0 \\
        \end{array}
      \right)dz,\]
      and
\[\hat{F}_{01}^{-1}d\hat{F}_{01}=\left(
              \begin{array}{cccc}
               0 & 0 & b_{13}' & b_{14}' \\
                0 &  0 & -b_{13}' & -b_{14}' \\
                 b_{13}'  &  b_{13}'   & 0 & 0 \\
                 b_{14}'  &  b_{14}'   & 0 & 0 \\
              \end{array}
            \right)dz,\]
   with
   \[b_{13}'=a_{13}\cos \varphi+a_{14}\sin \varphi,\ b_{14}'=-a_{13}\sin \varphi+a_{14}\cos \varphi.\]
Moreover,  one obtains
   \begin{equation*} \tilde{F}_{01}^{-1}\left(\hat{F}_{01}^{-1}d\hat{F}_{01}\right)\tilde{F}_{01}= \left(
        \begin{array}{cccc}
          0 & 0 & a_{13} & a_{14} \\
          0 & 0 & -a_{13} & -a_{14} \\
          a_{13} & a_{13} & 0 &0 \\
          a_{14} & a_{14} & 0& 0 \\
        \end{array}
      \right).
         \end{equation*}
Since $F_{01}=\hat{F}_{01}\tilde{F}_{01}$, one derives
\[F_{01}^{-1}F_{01z}=\tilde{F}_{01}^{-1}\left(\hat{F}_{01}^{-1}d\hat{F}_{01}\right)\tilde{F}_{01}+\tilde{F}_{01}^{-1}\tilde{F}_{01z}=A_1.\]
\end{proof}
\leftline{\em Proof of Theorem \ref{th-potential-light}:}\vspace{1mm}
Let $F(z,\bar{z},\lambda)=(\phi_1,\phi_{2},\phi_3,\phi_4,\psi_1,\ldots,\psi_n)$ be a frame of $f$ with the initial condition $F(0,0,\lambda)=I_{n+4}$. W.l.g, we may assume that \[Y_0=\phi_1-\phi_{2}\]
is the constant lightlike vector contained in $f$. As a consequence, we derive
\[\phi_{1z}=\phi_{2z}=a_{13}\phi_3+a_{14}\phi_4+\sqrt{2}\sum_{j=1}^{n}\beta_j\psi_j.\]
That is
\[(\phi_{1z},\phi_{2z})^t=
 \left(
                                 \begin{array}{ccccccc}
          0 & 0  & a_{13}& a_{14} & \sqrt{2}\beta_1 & \ldots & \sqrt{2}\beta_n  \\
          0 & 0  & a_{13}& a_{14} & \sqrt{2}\beta_1 & \ldots & \sqrt{2}\beta_n  \\
          \end{array}
                               \right)\cdot F^t.
\]
Comparing with
\[ F^{-1}F_z=\left(
                   \begin{array}{cc}
                     A_1 & B_1 \\
                     -B_1^tI_{1,3} & A_2 \\
                   \end{array}
                 \right),\]
we obtain
\[A_1=\left(
        \begin{array}{cccc}
          0 & 0 & a_{13} & a_{14} \\
          0 & 0 & -a_{13} & -a_{14} \\
          a_{13} & a_{13} & 0 & a_{34} \\
          a_{14} & a_{14} & -a_{34} & 0 \\
        \end{array}
      \right),\]
      and
      \[ \ B_1= \left(
      \begin{array}{ccc}
        \sqrt{2}\beta_1 & \ldots &\sqrt{ 2}\beta_n \\
        -\sqrt{2}\beta_1 & \ldots &  -\sqrt{ 2}\beta_n \\
        -k_1 & \ldots & -k_n \\
        -\hat{k}_1 & \ldots & -\hat{k}_n \\
      \end{array}
    \right).\]
Since $B_1^tI_{1,3}B_1=0$, \[\hat k_1=i k_1,\ldots, \hat k_n=i k_n, \hbox{ or } \hat k_1=-i k_1,\ldots, \hat k_n=-i k_n.\]
Similar to the discussion in Lemma 3.8 of Section 3 of \cite{DoWa1}, without loss of generality, we assume that on $\tilde{M}$,  $\hat k_1=i k_1,\ldots, \hat k_n=i k_n$.

For the computation of the normalized potential, we will apply Wu's formula (Theorem 4.23, Section 4.3 of \cite{DoWa1}, see also \cite{Wu}).
Let  $\delta_{1}=(\tilde{a}_{ij})$ denote the ``holomorphic part" of $ A_{1}$ with respect to the base point $z=0$, i.e., the part of the Taylor expansion of $A_1$ which is independent of $\bar{z}$.
Let $F_{01}$ be a solution to the equation
 \begin{equation}\label{F-01}F_{01}^{-1}dF_{01}=\delta_{1}dz,\  F_{01}|_{z=0}=I_4.\end{equation}
By Lemma \ref{lemma-ode}, $F_{01}$ is equal to
\[\left(
              \begin{array}{cccc}
                1+\frac{1}{2}(b_{13}^2+\hat{a}_{14}^2) & \frac{1}{2}(b_{13}^2+\hat{a}_{14}^2)  & b_{13}\cos \varphi- b_{14}  \sin\varphi &  b_{13}\sin \varphi + b_{14}  \cos\varphi \\
                 -\frac{1}{2}(b_{13}^2+b_{14}^2) & 1-\frac{1}{2}(b_{13}^2+b_{14}^2) & -b_{13}\cos \varphi+ b_{14}  \sin\varphi &  -(b_{13}\sin \varphi + b_{14}  \cos\varphi) \\
                 b_{13}  &  b_{13}   &\cos \varphi & \sin\varphi \\
                 b_{14}& b_{14}   &-\sin\varphi &\cos \varphi \\
              \end{array}
            \right) \]
with
\[\varphi =\int_0^z  \tilde{a}_{34} \ dz \]
and
\[ \ b_{13} =\int_0^z (\tilde{a}_{13}\cos \varphi+\tilde{a}_{14}\sin \varphi ) dz ,\  b_{14} = \int_0^z (-\tilde{a}_{13}\sin \varphi+\tilde{a}_{14}\cos \varphi) dz.\]
Let  $\delta_{2}$ denote the ``holomorphic part" of $ A_{2},$ with respect to the base point $z=0$, and let $F_{02}$ be a solution to the equation \[F_{02}^{-1}dF_{02}=\delta_{2}dz,\  F_{02}|_{z=0}=I_n.\] Let $\tilde{B}_1$ denote the holomorphic part of $B_1$. By Wu's formula (Theorem 4.23 of \cite{DoWa1}), the normalized potential can be represented in the form
 \begin{equation*}
 \begin{split}
\eta&=\lambda^{-1}\left(
                     \begin{array}{cc}
                       F_{01} & 0 \\
                       0 & F_{02} \\
                     \end{array}
                   \right)
\left(
                     \begin{array}{cc}
                       0 & \tilde{B}_1 \\
                       -\tilde{B}^{t}_1I_{1,3} & 0 \\
                     \end{array}
                   \right)\left(
                     \begin{array}{cc}
                       F_{01} & 0 \\
                       0 & F_{02} \\
                     \end{array}
                   \right)^{-1}dz\\
&=\lambda^{-1}\left(
                     \begin{array}{cc}
                       0 & \hat{B}_1 \\
                       -\hat{B}^{t}_1I_{1,3} & 0 \\
                     \end{array}
                   \right)dz,
                   \end{split}\end{equation*}
with             \begin{equation*}
\begin{split}            \hat{B}_{1} &=F_{01}\tilde{B}_{1}F_{02}^{-1} \\
&=F_{01}\cdot\left(
                                                                       \begin{array}{cccc}
                                                                         \tilde{f}_{11} & \tilde{f}_{12} & \ldots & \tilde{f}_{1n} \\
                                                                         -\tilde{f}_{11} & -\tilde{f}_{12} & \ldots & -\tilde{f}_{1n} \\
                                                                         -\tilde{f}_{31} & -\tilde{f}_{32} & \ldots & -\tilde{f}_{3n} \\
                                                                         -i\tilde{f}_{31} & -i\tilde{f}_{32} & \ldots & -i\tilde{f}_{3n} \\
                                                                       \end{array}
                                                                     \right)\cdot F_{02}^{-1}\\
                                                                      &=\left(
                                                                       \begin{array}{cccc}
                                                                         \hat{f}_{11} & \hat{f}_{12} & \ldots & \hat{f}_{1n} \\
                                                                         -\hat{f}_{11} & -\hat{f}_{12} & \ldots & -\hat{f}_{1n} \\
                                                                         -\hat{f}_{31} & -\hat{f}_{32} & \ldots & -\hat{f}_{3n} \\
                                                                         -i\hat{f}_{31} & -i\hat{f}_{32} & \ldots & -i\hat{f}_{3n} \\
                                                                       \end{array}
                                                                     \right).
\end{split}\end{equation*}

The statement that $\mathcal{F}$ is of finite uniton type with maximal uniton number $\leq 2$ comes from Lemma \ref{lemma-mini1} in Section 3.2.
\rightline{$\Box$}

\subsection{Proof of Theorem \ref{th-mini-spaceform}}

\ \\ \hspace{5mm}
{\em Proof of Theorem \ref{th-mini-spaceform}:}

Case (1) comes from Theorem \ref{th-potential-light} and Theorem \ref{th-minimal}.

Now we consider Case (2). Since $\mathcal{F}$ contains a constant timelike vector $e_0$. We can assume $|e_0|=1$ and it is time forward. Then there exists a transformation $T\in SO(1,n+3)$ transforming $e_0$ into $(1,0,\ldots,0)^t$ and transforming $\mathcal{F}$ into $T\mathcal{F}$. So without loss of generality, we assume $e_0=(1,0,\ldots,0)^t$. Let $F=(e_0,\hat{e}_0,e_1,e_2,\psi_1,\ldots,\psi_n)$ be a lift of $\mathcal{F}$. As a consequence, every entry of the first column and the first row of $\alpha=F^{-1}dF$ is zero.  The same holds for $F_{\lambda}$ when introducing the loop parameter into $F$.  Then using Wu's formula \cite{Wu} (see also Theorem 4.23, and Theorem 4.24 of \cite{DoWa1} for the Willmore case), we see that every entry of the first column and the first row of the normalized potential stays zero. Moreover, by Theorem 4.24, $\hat{B}_1$ also satisfies $\hat{B}_1^tI_{1,3}\hat{B}_1=0$, which yields \[\mathbf{v}_j^tI_{1,3}\mathbf{v}_l=0, ~~\hbox{ for all }~j,l=1,\ldots,n.\]
Formula \eqref{eq-potential-s} follows by a simple computation similar to the one in deriving the Weierstrass  representation of minimal surfaces in $\R^3$.

The converse part is straightforward. In fact, integrating $\eta$, we see that all the entries of the first column and the first row of $F_-$ are $0$, except the $(1,1)$-entry which is  $1$. Then performing an Iwasawa decomposition one observes that $F$ inherits the same property, that is, the harmonic map $\mathcal{F}$ contains $e_0=(1,0,\ldots,0)^t$ at every point.

Case (3) follows by a similar argument.

\rightline{$\Box$}

\section{Minimal surfaces in $\R^n$ as Willmore surfaces}

This and the next section aim to give by direct computations,  a concrete description of all Willmore surfaces corresponding to the first type of nilpotent Lie subalgebras in \cite{Wang-1}. To this end, since there are many lengthy and elementary computations, we  will divide the proof into  several technical lemmas, which will be stated in this section. The proofs of  these lemmas will be left to Section 4.

The basic idea in our computations is to express the normalized potentials by some strictly upper triangular matrix-valued 1-forms, since this will simply the computations substantially.  For this purpose, we need to transform the original group into a different one such that the matrix coefficients  of the potentials in \eqref{eq-min-np} will be upper   triangular matrices.  So we will first recall the Lie group isometry  in Section 3.1. Then we state five technical lemmas in Section 3.2.

\subsection{Preliminary}
To begin with, we first recall some basic notations and results. The detailed descriptions and proofs can be found in Section 3 of \cite{Wang-1}. We will retain the notation of \cite{Wang-1}.

Recall that $SO^+(1,n+3)=SO(1,n+3)_0$ is the connected subgroup of
\begin{equation*}SO(1,n+3):=\{ A\in Mat(n+4,\mathbb{R})\ |\ A^t I_{1,n+3}A=I_{1,n+3}, \det A=1\},
\end{equation*}
with
\[I_{1,n+3}=\hbox{diag}\{-1,1,\ldots,1\}.\]
The subgroup $K=SO^+(1,3)\times SO(n)$ is defined by the involution

 \begin{equation}\begin{array}{ll}
\sigma:    SO^+(1,n+3)&\rightarrow SO^+(1,n+3)\\
 \ \ \ \ \ \ \ A &\mapsto DAD^{-1},
\end{array}\end{equation}
where  $D=diag\{-I_4,I_n\}$.
 For simplicity, we assume that $n$ is even and $n+4=2m$.
We also have
\begin{equation}\label{eq-g}
G(n+4,\mathbb{C}):=\{A\in Mat(n+4,\mathbb{C})| A^tJ_{n+4}A=J_{n+4}, \det A=1
\},
\end{equation}
with  \[J_{n+4}=\left( j_{k,l}\right)_{(n+4)\times (n+4)}, \  j_{k,l}=\delta_{k+l,n+5} \hbox{ for all }1\leq k,l\leq n+4.\]
By Lemma 3.1 of \cite{Wang-1}, one obtains a Lie group isometry from $SO^+(1,2m-1,\C)$ into $G(2m,\C)$, defined by the following map
\begin{equation} \begin{array}{ccccc}
                   \mathcal{P}: &   SO^+(1,2m-1,\C)&  \rightarrow &  G(2m,\C)\\
              \     & A& \mapsto & \tilde{P}^{-1} A \tilde{P}\\
                 \end{array}
\end{equation}
with
\begin{equation}
\tilde{P}=\frac{1}{\sqrt{2}}\left(
    \begin{array}{cccccccc}
       1 &  &   &  &   &   &   & -1  \\
       1 &  &    &   &   &   & & 1   \\
        & -i &   &  &  &   &  i &   \\
       &  1&   &  &   &   &   1&   \\
       &  & \vdots  &  &     & \vdots   &   &   \\
       &  & \vdots  &  &     & \vdots    &   &   \\
       &  &    & -i & i  &    &   &   \\
       &  &   & 1 &  1 &    &   &   \\
    \end{array}
  \right).\end{equation}
  Under this isometry, we have that $\mathcal{P}\left(SO^+(1,2m-1)\right)$ is equal to the connected component of $\{F\in G(2m,\C)\ |\ F=S_{2m}^{-1}\bar{F}S_{2m}\}$
  containing $ I_{2m}$. Here
\begin{equation}  S_{2m}=\left(
        \begin{array}{ccccc}
          1 &   &   \\
           & J_{2m-2}  &    \\
           &   &    1\\
        \end{array}
      \right).
\end{equation}
Moreover, this induces  an involution of the loop group $\Lambda G(2m,\C)$:
 \begin{equation}\label{eq-def-tau} \begin{array}{ll}
\hat{\tau} :     \Lambda G(2m,\mathbb{C})&\rightarrow \Lambda G(2m,\mathbb{C})\\
 \ \ \ \ \ \ \ F   &\mapsto S_{2m}^{-1} \bar{F}S_{2m}
\end{array}\end{equation}
with $\mathcal{P}(\Lambda SO^+(1,2m-1))=\{F\in\Lambda G(2m,\mathbb{C})|
\hat{\tau}(F)=F\}$ as its fixed point set.
We also have that the image of $SO^+(1,3)\times SO(2m-4)$  under  $\mathcal{P}$ is of the form \begin{equation}\mathcal{P}\left((SO^+(1,3 )\times SO(2m-4))^{\C}\right)  =\{ F \in G(2m,\mathbb{C})\ |\ \ F=D_0^{-1}FD_0\ \}\end{equation}
with
\begin{equation}D_0=\tilde{P}^{-1}D \tilde{P}=\hbox{diag}\{-1,-1,I_{2m-4},-1,-1\}=\left(
                                 \begin{array}{ccccc}
                                   -1  &  &  & &\\
                                    & -1 &  & &\\
                                    &  & I_{2m-4} & & \\
                                    &  &   & -1 & \\
                                    &  &  &  & -1\\
                                 \end{array}
                               \right).
\end{equation}

\subsection{Technical Lemmas}
With the notations as above, we are able to state the  following lemmas:
\begin{lemma}\label{lemma-mini1} Let $\eta\in\Lambda^-\mathfrak{so}(1,n+3)_{\sigma}$ be the normalized potential defined on $\mathbb{D}$ in Theorem \ref{th-minimal}. Then
\[\mathcal{P}(\eta)=\lambda^{-1}\left(
        \begin{array}{ccc}
          0 & \tilde{f} & 0 \\
          0 & 0 & -\tilde{f}^{\sharp} \\
          0 & 0 & 0 \\
        \end{array}
      \right)dz,\ \hbox{ with } \tilde{f}\in Mat(2\times (2m-4),\C),\ \tilde{f}^{\sharp}:=J_{2m-4}\tilde{f}^tJ_2.\]
 \end{lemma}

 \begin{lemma}\label{lemma-mini2}Let $\eta$ be as in Lemma \ref{lemma-mini1}. Then $H=I_{2m}+\lambda^{-1}H_1+\lambda^{-2}H_2$ is the solution to \begin{equation}\label{eq-min-ini}H^{-1}dH=\mathcal{P}(\eta),\ ~ H|_{z=0}=I_{2m}
 \end{equation} with
\begin{equation}H_1=\left(
        \begin{array}{ccc}
          0 & f & 0 \\
          0 & 0 & -f^{\sharp} \\
          0 & 0 & 0 \\
        \end{array}
      \right),\ H_2=\left(
        \begin{array}{ccc}
          0 & 0 & g \\
          0 & 0 & 0 \\
          0 & 0 & 0 \\
        \end{array}
      \right),\end{equation}
\begin{equation} f= \int_0^z\tilde{f}dz,\ g=-\int_0^z(f\tilde{f}^{\sharp})dz.
\end{equation}
\end{lemma}

Note that if $\eta$ is derived from some strong Willmore map, then $\eta$ is meromorphic and also $H$, the integration of $\mathcal{P}(\eta)$, is meromorphic. If in Lemma \ref{lemma-mini1} we start from some normalized potential and want to construct a  strong Willmore map defined on $\mathbb{D}$, then we need to assume that $\eta$ is meromorphic and also that $H$ is meromorphic.

 \begin{lemma}\label{lemma-mini3} Retaining the assumptions and the notation of the previous lemmas, assume that $\mathcal{P}(\eta)$ is the normalized potential of some harmonic map, we obtain:

 The Iwasawa decomposition of $H$ is
 \[H=\tilde{F}\tilde{F}_+, \hbox{ with } \tilde{F}\in \mathcal{P}(\Lambda SO^+(1,2m-1)_{\sigma})\subset \Lambda G(2m,\C)_{\sigma},\ \tilde{F}_+\in\Lambda^+ G(2m,\C)_{\sigma}.\]
And $\tilde{F}$ is given by (see also (3.7) of \cite{Wang-1})
\begin{equation}\label{eq-m-c-min}\tilde{F}=H\hat{\tau}(W)L_0^{-1}.
\end{equation}
Here
$W$, $W_0$ and $L_0$ are the solutions to the matrix equations
$$ \hat{\tau}(H)^{-1}H= WW_0\hat{\tau}(W)^{-1},\  ~~ W_0=\hat{\tau}(L_0)^{-1}L_0$$
with $$W=I_{2m}+\lambda^{-1}W_1+\lambda^{-2}W_2,\ ~ W_1=\left(
        \begin{array}{ccc}
          0 & u & 0 \\
         0 & 0 & -u^{\sharp} \\
          0 &  0 & 0 \\
        \end{array}
      \right),\ ~ W_2=\left(
        \begin{array}{ccc}
         0 & 0 & \hat{g} \\
          0 & 0 & 0 \\
          0 & 0 & 0 \\
        \end{array}
      \right),$$
and $$W_0=\left(
        \begin{array}{ccc}
          a & 0 & b\\
          0 & q & 0 \\
          0 & 0 & d \\
        \end{array}
      \right),
     \  ~~L_0=\left(
                                \begin{array}{ccc}
                                  l_1 & 0 & l_2 \\
                                  0 & l_0 & 0\\
                                  0 & 0 & l_4 \\
                                \end{array}
                              \right),
      ~\hbox{ with } l_1,\ l_4 \hbox{ upper triangular} .$$
Moreover, we have
\begin{subequations}  \label{eq-iwa:1}
\begin{align}
 &d=I_2+E_4\bar{f}^{t\sharp} {f}^{ \sharp}+ E_4\bar{g}^{t}E_1g,          \label{eq-iwa:1A} \\
  & u^{\sharp}d =f^{\sharp}-\bar{f}^tE_1g,       \label{eq-iwa:1B} \\
  & q + u^{\sharp}dE_4\bar{u}^{\sharp t} =I_{2m-4}+\bar{f}^tE_1f,   \label{eq-iwa:1C}\\
& a+ uq\bar{u}^tE_1+ gE_4\bar{\hat{d}}^t\bar{g}^tE_1=I_2, \label{eq-iwa:1D}\\
& b+uq\bar{u}^{t}E_2+gE_4\bar{\hat{d}}^t\bar{g}^tE_2 =E_2\bar{f}^{t \sharp}{f}^{ \sharp} +E_2\bar{g}^{t}E_1g, \label{eq-iwa:1E}\\
 & uq-gE_4\bar{u}^{\sharp t}= f.  \label{eq-iwa:1F}
\end{align}
\end{subequations}
Here $E_1,$ $E_2$, $E_3$ and $E_4$ are defined as
\begin{equation}\label{eq-eij}E_1=\left(
               \begin{array}{cc}
                 0 & 0 \\
                 0 & 1 \\
               \end{array}
             \right),\ E_2=E_3^t=\left(
               \begin{array}{cc}
                 0 & 1 \\
                 0 & 0 \\
               \end{array}
             \right),\ E_4=\left(
               \begin{array}{cc}
                 1 & 0 \\
                 0 & 0 \\
               \end{array}
             \right).
\end{equation}
 \end{lemma}

 \begin{remark} \

  1. Since in Lemma \ref{lemma-mini3} the matrices $f$ and $g$, whence also $f^{\sharp}$, are given, equation \eqref{eq-iwa:1A} determines $d$, where $d$ is invertible (certainly  true for small $z$ close to $z=0$). Then equation \eqref{eq-iwa:1B} determines $u^{\sharp}$, hence $u$. Inserting this into \eqref{eq-iwa:1C} results in determining $q$. Inserting what we have so far into \eqref{eq-iwa:1D} determines $a$ and similarly from \eqref{eq-iwa:1E} we obtain $b$. The last equation, \eqref{eq-iwa:1F}, is a consequence of the previous equations. Therefore, the only condition for the solvability of the system of equations is the invertibility of $d$.

 2. If $f$ and $g$ are rational functions of $z$, the invertibility of $d$ is satisfied locally, whence on an open dense subset due to the rational expression in $z,\bar{z}$.
 \end{remark}

\begin{lemma}\label{lemma-mini4}Retaining the assumptions and the notation of the previous lemmas, the Maurer-Cartan form  of  $\tilde{F}$ in \eqref{eq-m-c-min} is of the form
\begin{equation}\label{eq-m-c-min2}\tilde\alpha_{\mathfrak{p}}'=\lambda^{-1}\left(
                                \begin{array}{ccc}
                                  0 & \tilde{b} & 0 \\
                                  0 & 0 & -\tilde{b}^{\sharp}\\
                                  0 & 0 & 0 \\
                                \end{array}
                              \right)dz,\ ~~\tilde\alpha_\mathfrak{k}'=\left(
                                \begin{array}{ccc}
                                  a_1 & 0 & a_2 \\
                                  0 & a_0 & 0\\
                                  0 & 0 & a_4 \\
                                \end{array}
                              \right)dz.
\end{equation}
 \end{lemma}

\begin{lemma}\label{lemma-mini5} Let $\mathcal{F}:M\rightarrow SO^+(1,2m-1)/SO^+(1,3)\times SO(2m-4)$ be a strongly conformally harmonic map with an extended frame $F$. If the  Maurer-Cartan form of $\tilde{F}=\mathcal{P}(F)$ is of the form
\eqref{eq-m-c-min2} in $\mathfrak{g}(2m,\C)$, then $\mathcal{F}$ contains a constant light-like vector. Therefore, if $\mathcal{F}$ is the conformal Gauss map of some Willmore map $y$, $y$ is M\"{o}bius equivalent to a minimal surface in $\R^{2m}$.\end{lemma}

Lemma \ref{lemma-mini1} and Lemma \ref{lemma-mini2} can be verified by straightforward computations. And the other lemmas will be proven in the following section.\ \vspace{2mm}

\leftline{\em Proof of Theorem \ref{th-minimal}:}\vspace{1mm}
Combination of the above five lemmas provides the proof of  Theorem \ref{th-minimal}.

\rightline{$\Box$}

Note  that Corollary \ref{cor-min} already follows from  the above three lemmas. The fact that $\eta_{-1}$ takes values in a nilpotent Lie subalgebra of rank $2$ comes from Theorem 2.6 and Lemma 3.5 of \cite{Wang-1}.
 \begin{remark}\

  1. For a general procedure for the computation of Iwasawa decompositions for algebraic loops, or more generally for rational loops, see $\S I.2$ of \cite{CG}, where they provided a somewhat constructive method to do the decompositions for such loops. One may also compare our treatment here with the ones for CMC surfaces in \cite{Do-Ha}.

  2. For the  theoretical description of the loop group of a non-compact Lie group, we refer to \cite{Ba-Do} and \cite{Ke1}. Another concrete discussion of  $\Lambda SU(1,1)$ can be found in \cite{B-R-S}.

  3. Recently  there have been several publications concerning harmonic maps into compact Lie groups and compact symmetric spaces which have used methods which are different from ours, see
  \cite{FSW}, \cite{Co-Pa}, \cite{S-W} and reference therein. Most of their work basically follows the techniques developed by  Uhlenbeck \cite{Uh} and Segal \cite{Segal}. We also  note that in \cite{S-W}, a converse  procedure is used  for the computations for the Iwasawa decompositions of  loops in   $\lambda_{alg}U(n)^{\C}$, which provides another way to do the concrete Iwasawa decomposition  for algebraic loops.
\end{remark}

\section{Iwasawa decompositions }

In this section we first provide the proof of Lemma \ref{lemma-mini3}, which corresponds to the Iwasawa decompositions. Then, we derive the M-C forms by the information from the explicit Iwasawa decompositions, which yields the geometric descriptions of the corresponding harmonic maps.

\subsection{Iwasawa decompositions and Lemma \ref{lemma-mini3}}\ \\
{\em Proof of Lemma \ref{lemma-mini3}:}  \vspace{2mm}

The first question is the existence of the Iwasawa decomposition for $H$. This is guaranteed by the existence of an Iwasawa decomposition on an open subset containing the identity and $H|_{z=0}=I$ (see Theorem 4.1 of \cite{DoWa1}, Theorem 2.3 of \cite{DoWa2}, also \cite{Ke1}). So certainly near $z=0$ we have $H=\tilde{F}\tilde{F}_+$. Since $\hat\tau(\tilde{F})=\tilde{F}$, the maximal and minimal powers of $\lambda$ of $\tilde{F}$ are $\lambda^2$ and $\lambda^{-2}$ respectively. Hence in $\tilde{F}_+$ the maximal power of $\lambda$ is at most $2$. Moreover, from the definition of $G(n+4,\mathbb{C})$ we infer that also $(\tilde{F}_+)^{-1}$ only contains the powers of $\lambda$ from $-2$ to $2$.
Moreover,
\[I_{n+4}=\hat\tau(\tilde{F})^{-1}\tilde{F}=\left(\hat\tau(\tilde{F}_+^{-1})\right)^{-1}\hat\tau(H)^{-1}H\tilde{F_+}^{-1}\]
implies that
\[\hat\tau(H)^{-1}H=\hat\tau(\tilde{F}_+^{-1})\tilde{F_+}.\]
 Let's write \[\hat\tau(\tilde{F}_+^{-1})=W\hat\tau(L_0)^{-1}\] with
 \[W= I_{2m}+\lambda^{-1}W_1+\lambda^{-2}W_2,\]
 where
 \[ \ ~~  W_1=\left(
        \begin{array}{ccc}
          0 & u & 0 \\
         -v^{\sharp} & 0 & -u^{\sharp} \\
          0 &  v & 0 \\
        \end{array}
      \right).\]
 Set $W_0=\hat{\tau}(L_0)^{-1}L_0$ with
      \[ W_0=\left(
        \begin{array}{ccc}
          a & 0 & b \\
          0 & q & 0 \\
          c & 0 & d \\
        \end{array}
      \right), ~~ W_0^{-1}=\left(
        \begin{array}{ccc}
          \hat{a} & 0 & \hat{b} \\
          0 & q^{-1} & 0 \\
          \hat{c} & 0 & \hat{d} \\
        \end{array}
      \right), \hbox{ and }~~~  L_0=\left(
                                \begin{array}{ccc}
                                  l_1 & 0 & l_2 \\
                                  0 & l_0 & 0\\
                                  0 & 0 & l_4 \\
                                \end{array}
                              \right).\]
With these notations,    we obtain
$$ \hat{\tau}(H)^{-1}H=WW_0\hat\tau(W)^{-1}.$$
For explicit computations we recall
from \eqref{eq-def-tau} (See also (3.7) of \cite{Wang-1}) that for any $F\in G(2m,\mathbb{C})$ we have
\begin{equation}\label{eq-def-tau-inv}\hat\tau(F)^{-1}=\hat{J}_{2m}\bar{F}^t\hat{J}_{2m}, \end{equation}
where
\begin{equation}\hat{J}_{2m}=S_{2m}J_{2m}=\left(
  \begin{array}{ccc}
    E_1 & 0 & E_2 \\
    0 & I_{2m-4} & 0 \\
    E_3 &0 & E_4 \\
  \end{array}
\right).\end{equation}
Recall that \eqref{eq-eij}
\begin{equation*}E_1=\left(
               \begin{array}{cc}
                 0 & 0 \\
                 0 & 1 \\
               \end{array}
             \right),\ E_2=E_3^t=\left(
               \begin{array}{cc}
                 0 & 1 \\
                 0 & 0 \\
               \end{array}
             \right),\ E_4=\left(
               \begin{array}{cc}
                 1 & 0 \\
                 0 & 0 \\
               \end{array}
             \right).
\end{equation*}
As a consequence, we derive
\begin{equation}\label{eq-wh}\left\{
\begin{split}
W_2W_0 &=H_2,\\
W_1W_0+W_2W_0 (\hat{J}_{2m}\bar{W}_1^t\hat{J}_{2m})&=H_1+ (\hat{J}_{2m}\bar{H}_1^t\hat{J}_{2m})H_2,\\
\widehat{W} &=I_{2m}+ (\hat{J}_{2m}\bar{H}_1^t\hat{J}_{2m})H_1+(\hat{J}_{2m}\bar{H}_2^t\hat{J}_{2m})H_2.
\end{split}
\right.
\end{equation}
Here
\[\widehat{W}=W_0+W_1W_0 (\hat{J}_{2m}\bar{W}_1^t\hat{J}_{2m})+W_2W_0(\hat{J}_{2m}\bar{W}_2^t\hat{J}_{2m}) .\]
The first equation of \eqref{eq-wh} yields
$$W_2=H_2W_0^{-1}=\left(
        \begin{array}{ccc}
          g\hat{c} & 0 & g\hat{d} \\
          0 & 0 & 0 \\
          0 & 0 & 0 \\
        \end{array}
      \right).$$
Next we evaluate the second matrix equation of \eqref{eq-wh}. We compute
\begin{equation*}
\begin{split}
&H_1+ (\hat{J}_{2m}\bar{H}_1^t\hat{J}_{2m})H_2=\left(
        \begin{array}{ccc}
          0 & f & 0 \\
          0 & 0 & -f^{\sharp}+\bar{f}^tE_1g \\
          0 & 0 & 0 \\
        \end{array}
      \right),\\
      &\hat{J}_{2m}\bar{W}_1^t\hat{J}_{2m}=\left(
        \begin{array}{ccc}
          0 & -E_1\bar{v}^{\sharp t}-E_2\bar{u}^{\sharp t} & 0 \\
          \bar{u}^tE_1+\bar{v}^{t}E_3 & 0 &   \bar{u}^tE_2+\bar{v}^{t}E_4 \\
          0 & -E_3\bar{v}^{\sharp t}-E_4\bar{u}^{\sharp t}  & 0 \\
        \end{array}
      \right),\\
 &H_2 (\hat{J}_{2m}\bar{W}_1^t\hat{J}_{2m})=\left(
        \begin{array}{ccc}
          0 &-gE_3\bar{v}^{\sharp t}-gE_4\bar{u}^{\sharp t} & 0 \\
       0& 0 & 0 \\
          0 &0 & 0 \\
        \end{array}
      \right),\\ & W_1W_0=\left(
        \begin{array}{ccc}
          0 & uq & 0 \\
       -v^{\sharp}a-u^{\sharp}c& 0 & -v^{\sharp}b-u^{\sharp}d \\
          0 &vq & 0 \\
        \end{array}
      \right).
      \end{split}
\end{equation*}
As a consequence we read off the equations
\begin{equation*}
vq=0,\ -v^{\sharp}a-u^{\sharp}c=0,\ uq-gE_3\bar{v}^{\sharp t}-gE_4\bar{u}^{\sharp t}= f,\
 -v^{\sharp}b-u^{\sharp}d =-f^{\sharp}+\bar{f}^tE_1g.
\end{equation*}
Since $q$ is invertible, $v=0$. Therefore these equations reduce to the following ones:
\begin{equation*}
v=0,\ u^{\sharp}c=0,\ uq-gE_4\bar{u}^{\sharp t}= f,\  u^{\sharp}d =f^{\sharp}-\bar{f}^tE_1g.
\end{equation*}
Similarly, for the third equation of \eqref{eq-wh}, we first compute matrix expressions for the summands:
\begin{equation*}
\begin{split}
(\hat{J}\bar{H}_1^t\hat{J})H_1&=\left(
        \begin{array}{ccc}
          0& 0  & E_2\bar{f}^{t \sharp}{f}^{ \sharp} \\
         0& \bar{f}^tE_1f & 0  \\
          0& 0 & E_4\bar{f}^{t\sharp} {f}^{ \sharp} \\
        \end{array}
      \right),\\
      (\hat{J}\bar{H}_2^t\hat{J})H_2& =\left(
        \begin{array}{ccc}
          0& 0  & E_2\bar{g}^{t}E_1g \\
         0& 0 & 0  \\
          0& 0 & E_4\bar{g}^{t}E_1g \\
        \end{array}
      \right),\\
W_1W_0 (\hat{J}\bar{W}_1^t\hat{J})&=\left(
        \begin{array}{ccc}
          uq\bar{u}^tE_1& 0  & uq\bar{u}^{t}E_2 \\
         0& u^{\sharp}dE_4\bar{u}^{\sharp t} & 0  \\
          0& 0 & 0 \\
        \end{array}
      \right),\\
\hat{J}\bar{W}_2^t\hat{J}&=\left(
        \begin{array}{ccc}
       E_1\bar{\hat{c}}^t\bar{g}^tE_1+
       E_2\bar{\hat{d}}^t\bar{g}^tE_1 & 0  &    E_1\bar{\hat{c}}^t\bar{g}^tE_2+
       E_2\bar{\hat{d}}^t\bar{g}^tE_2\\
         0& 0& 0  \\
         E_3\bar{\hat{c}}^t\bar{g}^tE_1+
       E_4\bar{\hat{d}}^t\bar{g}^tE_1 & 0  &    E_3\bar{\hat{c}}^t\bar{g}^tE_2+
       E_4\bar{\hat{d}}^t\bar{g}^tE_2 \\
        \end{array}
      \right),\\
W_2W_0 (\hat{J}_{2m}\bar{W}_2^t\hat{J})&=\left(
        \begin{array}{ccc}
         gE_3\bar{\hat{c}}^t\bar{g}^tE_1+
       gE_4\bar{\hat{d}}^t\bar{g}^tE_1 & 0  &    gE_3\bar{\hat{c}}^t\bar{g}^tE_2+
       gE_4\bar{\hat{d}}^t\bar{g}^tE_2 \\     0& 0& 0  \\
          0& 0 & 0 \\
        \end{array}
      \right).\\
      \end{split}
\end{equation*}
Substituting these expressions into the third matrix equation of \eqref{eq-wh}, we derive that $c=0.$
Therefore we have
\begin{equation*}\left\{
\begin{split}
&c=0,\ a+ uq\bar{u}^tE_1+ gE_4\bar{\hat{d}}^t\bar{g}^tE_1=I_2,\\
 &b+uq\bar{u}^{t}E_2+gE_4\bar{\hat{d}}^t\bar{g}^tE_2 =E_2\bar{f}^{t \sharp}{f}^{ \sharp} +E_2\bar{g}^{t}E_1g,\\
& q + u^{\sharp}dE_4\bar{u}^{\sharp t} =I_{2m-4}+\bar{f}^tE_1f,\ d=I_{2}+E_4\bar{f}^{t\sharp} {f}^{ \sharp}+ E_4\bar{g}^{t}E_1g.\\
\end{split}
\right.
\end{equation*}
Summing up, we obtain \eqref{eq-iwa:1}.

The only thing left to prove Lemma \ref{lemma-mini3} is the statement of the form of $L_0$.
To derive this, we first consider $W_0$.  Recall that
\[ W_0=W^{-1}\hat\tau(H)^{-1}H\hat\tau(W)=\hat\tau(W_0)^{-1}=\hat{J}_{2m}\bar{W}_0^t\hat{J}_{2m},\]
and
$ W_0\in G(2m ,\mathbb{C})$. Hence, in particular, we also have
\[W_0^tJ_{2m}W_0=J_{2m}.\]
A direct computation using these equations shows
\begin{equation*}
 \tilde{W}_0=\left(
                  \begin{array}{cc}
                    a & b \\
                    c & d \\
                  \end{array}
                \right)=\left(
    \begin{array}{cccc}
      a_{11} & a_{12} & a_{13} & a_{14} \\
      0 & a_{22} & 0 & \bar{a}_{12} \\
      0 & 0 & a_{33} &\bar{a}_{13}  \\
      0 & 0 & 0 &
      \bar{a}_{11} \\
    \end{array}
  \right),\end{equation*}
with
\begin{equation*}
 |a_{11}|^2=1,\ a_{22}a_{33}=1,\
 \bar{a}_{13}=-\frac{a_{12}\bar{a}_{11}}{a_{22}},\ \bar{a}_{14}=\frac{|a_{12}|^2}{a_{22}}.
\end{equation*}
Set \begin{equation*}
\tilde{l}_0=\left(
                  \begin{array}{cc}
                    l_1 & l_3 \\
                    0 & l_4 \\
                  \end{array}
                \right)=\left(
    \begin{array}{cccc}
      l_{11} & l_{12} & l_{13} & l_{14} \\
      0 & l_{22} & 0 & l_{24} \\
      0 & 0 & l_{33} &l_{34}  \\
      0 & 0 & 0 & l_{44} \\
    \end{array}
  \right),\end{equation*}
with $l_{ij}$ satisfying
\[a_{11}=l_{11}^2,\ a_{12}=l_{11}l_{12}+\bar{l}_{24}l_{22},\ a_{22}=|l_{22}|^2,\]
and
\[l_{14}=-\frac{l_{12}l_{13}}{l_{11}},\ l_{24}=-\frac{l_{13}l_{22}}{l_{11}},\ l_{33}=\frac{1}{l_{22}},\ l_{34}=-\frac{l_{12}}{l_{11}l_{22}},\  l_{44}=\frac{1}{l_{11}}.\]
Moreover, let $l_0$ be a solution to
\[q=\bar{l}_0l_0.\]
Applying \eqref{eq-def-tau-inv}, it is a straightforward computation to verify that
\[\hat{\tau}(L_0)^{-1}L_0=\hat{J}_{2m}\bar{L}_0^t\hat{J}_{2m}L_0=W_0.\]
This finishes the proof of Lemma \ref{lemma-mini3}.
 \hfill$
\Box $

\begin{remark}  In the proof above  the splitting $q=\bar{l}_0l_0$ implies a strong restriction on $q$. This condition will always be satisfied near the identity. While in general it may happen  that one needs some middle term in the splitting $q=\bar{l}_0l_0$ due to the non-globality  of the Iwasawa splitting in our case. Actually the situation is much more complicated (for a similar situation, see \cite{B-R-S}). Since we have two open Iwasawa cells by Section 6 \cite{DoWa1}, Theorem 6.7, it may happen that $q$ starts at $I$ in the first open Iwasawa cell $I_1$ for some $z_0$ and   moves forward to  the boundary between the  two open Iwasawa cells $I_1$ and $I_2$. It could touch the boundary and return to $I_1$ or it could cross into $I_2$. What this means geometrically has been investigated only to a very small extent so far, but would seem to be a highly interesting project. But there are certainly cases where everything works just fine. See the example below.
\end{remark}

\begin{example}
 Let $2m=6$ and assume
$$f=\left(
      \begin{array}{cc}
        f_1 & f_2 \\
        f_3 & f_{4} \\
      \end{array}
    \right),\ g=\left(
      \begin{array}{cc}
        g_1 & g_2 \\
        g_3 & g_{4} \\
      \end{array}
    \right),$$
    with $f_j$, $g_j$, $j=1,\ldots,4$, meromorphic functions.
It is easy to derive that
\[g_1+g_4=-f_1f_4-f_2f_3,\ g_2=-f_1f_2,\ g_3=-f_3f_4\]
holds. From the last equation in \eqref{eq-iwa:1A}, and from
 \begin{equation*} \bar{f}^{t\sharp} {f}^{ \sharp}=J\bar{f}f^tJ=\left(
      \begin{array}{cc}
       |f_3|^2+|f_4|^2& \bar{f}_3f_1 + \bar{f}_4f_2 \\
           \bar{f}_1f_3 + \bar{f}_2f_4&    |f_1|^2+|f_2|^2 \\
      \end{array}
    \right),\  \end{equation*}
    and
    \begin{equation*} \bar{g}^{t}E_1g=\left(
      \begin{array}{cc}
        |g_3|^2 & \bar{g}_3g_4 \\
        \bar{g}_4g_3 & |g_4|^2 \\
      \end{array}
    \right),\end{equation*}
we obtain
\begin{equation*}d=\left(
      \begin{array}{cc}
        (1+|f_3|^2)(1+|f_4|^2) &  \bar{f}_3f_1 + \bar{f}_4f_2 \\
        0 & 1 \\
      \end{array}
    \right)=\left(
      \begin{array}{cc}
        d_1 & d_2 \\
        0 & 1 \\
      \end{array}
    \right).
\end{equation*}
Therefore
\begin{equation*}
a_1=1,\ a_5^{-1}=(1+|f_3|^2)(1+|f_4|^2),\ a_2=-\frac{f_1\bar{f}_3 + f_2\bar{f}_4+\bar{g}_3g_4 }{1+|f_3|^2+|f_4|^2+|g_3|^2}.
\end{equation*}
Moreover, we have
$$d^{-1}=\left(
      \begin{array}{cc}
        d_1^{-1} & -d_1^{-1}d_2 \\
        0 & 1 \\
      \end{array}
    \right).$$
From \eqref{eq-iwa:1B} we obtain
\begin{equation*}
u^{\sharp}=\left(
      \begin{array}{cc}
        f_4-\bar{f}_3g_3 & f_2-\bar{f}_3g_4  \\
       f_3-\bar{f}_4g_3 & f_1-\bar{f}_4g_4  \\
      \end{array}
    \right)\cdot d^{-1}=\left(
      \begin{array}{cc}
        u_4 & u_2 \\
        u_3 & u_{1} \\
      \end{array}
    \right)
\end{equation*}
with
\begin{equation*}
u_1=\frac{ f_1-f_2f_3\bar{f}_4-\bar{f}_4g_4}{1+|f_3|^2},\ \
u_2=\frac{ f_2-f_1\bar{f}_3f_4-\bar{f}_3g_4}{1+|f_4|^2},\ \
u_3=\frac{f_3}{1+|f_3|^2},\ \
u_4=\frac{f_4}{1+|f_4|^2}.\ \
 \end{equation*}
Moreover, since
\[\bar{f}^tE_1f=\left(
      \begin{array}{cc}
        \bar{f}_3f_3 & \bar{f}_3f_4  \\
        \bar{f}_4f_3 & \bar{f}_4f_4  \\
      \end{array}
    \right),\]\[ u^{\sharp}dE_4\bar{u}^{\sharp t}=\left(
      \begin{array}{cc}
        |u_4|^d_1 & u_4\bar{u}_3d_1 \\
        u_3\bar{u}_4d_1 & |u_3|^2d_1 \\
      \end{array}
    \right)=\left(
      \begin{array}{cc}
        \frac{|f_4|^2(1+|f_3|^2)}{1+|f_4|^2} & \bar{f}_3f_4  \\
        \bar{f}_4f_3 & \frac{|f_3|^2(1+|f_4|^2)}{1+|f_3|^2}  \\
      \end{array}
    \right),\]
from \eqref{eq-iwa:1C}, we have
\begin{equation*}
q=\left(
      \begin{array}{cc}
        \frac{1+|f_3|^2}{1+|f_4|^2} & 0\\
        0 & \frac{1+|f_4|^2}{1+|f_3|^2} \\
      \end{array}
    \right).
\end{equation*}
 Clearly, such expression can be written in the form $q=\bar{l}_1^tl_1$, and the above computations provides a global solution of equation \eqref{eq-iwa:1}. This shows that for the case $2m=6$ there exists a global Iwasawa splitting for harmonic maps with normalized potential of the form in Theorem \ref{th-minimal}.
\end{example}

\subsection{The Maurer-Cartan form of $\tilde{F}$}

\ \\
{\em Proof of Lemma \ref{lemma-mini4}:}\vspace{2mm}

To take a look at the Maurer-Cartan form of $\tilde{F}$, we first recall that
$$H^{-1}H_z=\lambda^{-1} \eta_{-1}=\lambda^{-1}\left(
                                \begin{array}{ccc}
                                  0 & \tilde{f}  & 0 \\
                                  0 & 0 & -\tilde{f}^{\sharp} \\
                                  0 & 0 & 0 \\
                                \end{array}
                              \right).
$$
Since
\[\tilde{F}^{-1}\tilde{F}_z=\lambda^{-1} L_0\hat{\tau}(W)^{-1}\eta_{-1}\hat{\tau}(W)L_0^{-1}+L_0\hat{\tau}(W)^{-1}\hat{\tau}(W)_zL_0^{-1}
+L_0(L_0^{-1})_z,\]
and \[\hat\tau(W)=I_{2m}+\lambda\hat\tau(W_1)+\lambda^2\hat\tau(W_2),\]
we may assume that
\[\tilde{F}^{-1}\tilde{F}_zdz=\lambda^{-1} \tilde\alpha_{-1}+
\tilde\alpha_0+\lambda\tilde{\alpha}_1+\lambda^2\tilde{\alpha}_2+\lambda^3\tilde{\alpha}_3+\lambda^4\tilde{\alpha}_4,\
\tilde\alpha_{-1}=L_0\eta_{-1}L_0^{-1}dz.\]
On the other hand, the reality condition yields
\[\hat\tau(\tilde{F}^{-1}d\tilde{F})=\tilde{F}^{-1}d\tilde{F}
,\] and
\[\tilde{F}^{-1}\tilde{F}_zdz=\lambda^{-1}\tilde{\alpha}_{-1}+\tilde{\alpha}_0+\lambda\tilde{\alpha}_{1}\]
with \[\tilde\alpha_1=\hat\tau(\tilde\alpha_{-1}),\ ~\hbox{ and }\ ~ \tilde\alpha_0=\hat\tau(\tilde\alpha_0).\]
Moreover, a straightforward computation yields
\[\tilde\alpha_0'=L_0[\eta_{-1},\hat\tau(W_1)]L_0^{-1}+L_0(L_0^{-1})_zdz.\]
Since $L_0$ is an upper triangular block matrix, $\tilde\alpha'$ is of the desired form stated in \eqref{eq-m-c-min2}.

\hfill $ \Box $
\ \\\\
{\em Proof of Lemma \ref{lemma-mini5}:}
\vspace{2mm}

  By Lemma \ref{lemma-mini4}, there exists a frame $\tilde{F}$ such that $\tilde\alpha'$ is of the form stated in \eqref{eq-m-c-min2}. By (3.3), (3.5) and (3.7), we obtain
  \[\alpha'=\mathcal{P}^{-1}(\tilde \alpha')=\left(
                   \begin{array}{cc}
                     A_1 & \lambda^{-1}B_1 \\
                     -\lambda^{-1}B_1^tI_{1,3} & A_2 \\
                   \end{array}
                 \right) d z,\]
                 with
                 \[A_1=\left(
         \begin{array}{ccccc}
         0 &  a_{12} &  a_{13} & a_{14} \\
          a_{12}  & 0 &  a_{13} & a_{14} \\
            a_{13} & -a_{13} & 0 & a_{34} \\
            a_{14}  & -a_{14}& -a_{34}  & 0 \\
         \end{array} \right),\]
         and
         \[\ B_1 = \left(
                                          \begin{array}{ccccccc}
                                            h_{11} & \hat{h}_{11} &  h_{12} & \hat{h}_{12} &\ldots &  h_{1,m-2}& \hat{h}_{1,m-2} \\
                                             h_{11} & \hat{h}_{11} &  h_{12}& \hat{h}_{12}&\ldots &  h_{1,m-2} & \hat{h}_{1,m-2}\\
                                            h_{31}& \hat{h}_{31} &  h_{32}& \hat{h}_{32} &\ldots &  h_{3,m-2}& \hat{h}_{3,m-2} \\
                                            ih_{31}& i\hat{h}_{31} &  ih_{32}& i\hat{h}_{32}&\ldots &  ih_{3,m-2}& i\hat{h}_{3,m-2} \\
                                          \end{array}
                                        \right).\]
Set \[F=\mathcal{P}^{-1}(\tilde{F})=(\phi_1,\phi_2,\phi_3,\phi_4,\psi_1,\ldots,\psi_{2m-4}).\]
 We have now
\begin{equation*}
\left\{
                  \begin{array}{cc}
                  \phi_{1z}=& a_{12}\phi_2+ a_{13}\phi_3+ a_{14}\phi_4+h_{11}\psi_1+\ldots+\hat{h}_{1,m-2}\psi_{2m-4},\\
                  \phi_{2z}=& a_{12}\phi_1-a_{13}\phi_3- a_{14}\phi_4-h_{11}\psi_1-\ldots-\hat{h}_{1,m-2}\psi_{2m-4}.\\
                      \end{array}\right.
 \end{equation*}
This indicates
 \[(\phi_1+\phi_2)_z=a_{12}(\phi_1+\phi_2).\]
 Since $\phi_1+\phi_2$ is a non-zero real vector-valued function, let $\phi_{01}$ be the first coordinate of $\phi_1+\phi_2$, it is straightforward to verify that $\frac{1}{\phi_{01}}(\phi_1+\phi_2)$ is a non-zero constant lightlike vector. As a consequence, a well-known fact states that if $\mathcal{F}$ is the conformal Gauss map of a Willmore surface $y$,
 $y$ is M\"{o}bius equivalent to some minimal surface in $\mathbb{R}^{2m-2}$  \cite{Helein}.

 \hfill$\Box $

\vspace{2mm}

{\bf Acknowledgements}: The author is thankful to Professor Josef. Dorfmeister, Professor Changping Wang and Professor Xiang Ma for their suggestions and encouragement.   This work is supported by the Project 11201340 of NSFC.

 \bigskip
\address{ % First Author
Department of Mathematics \\
Tongji University \\
 Siping Road 1239,Shanghai, 200092, \\
 P. R. China
 }
 { netwangpeng@tongji.edu.cn}
 \end{document}